\documentclass[a4paper,12pt,centertags]{amsart}
\usepackage[utf8]{inputenc}
\usepackage{mathrsfs}
\usepackage{eulervm}

\usepackage[T1]{fontenc}
\usepackage[colorlinks,pdfdisplaydoctitle]{hyperref}
\usepackage[a4paper,asymmetric]{geometry}
\newtheorem{theorem}{Theorem}[section]
\newtheorem{lemma}[theorem]{Lemma}
\newtheorem{proposition}[theorem]{Proposition}

\theoremstyle{definition}
\newtheorem{definition}[theorem]{Definition}

\theoremstyle{remark}
\newtheorem{remark}[theorem]{Remark}
\newtheorem{example}[theorem]{Example}
\numberwithin{equation}{section}

\DeclareMathAlphabet{\mathbbm}{U}{bbm}{m}{n}
\newcommand{\uno}{\mathbbm{1}}
\newcommand{\e}{\operatorname{e}}  
\newcommand{\im}{\mathrm{i}}       
\newcommand{\N}{\mathbf{N}}        
\newcommand{\Z}{\mathbf{Z}}        
\newcommand{\R}{\mathbf{R}}        
\newcommand{\loc}{{\mathrm{loc}}}
\newcommand{\map}{\mathscr{F}}
\newcommand{\mapV}{\mathscr{V}}
\newcommand{\X}{\mathcal{X}}
\newcommand{\B}{\mathcal{B}}
\newcommand{\vmo}{\mathcal{Z}}
\newcommand{\eps}{\epsilon}
\newcommand{\term}[1]{\text{\textcircled{\sf\footnotesize #1}}}
\hypersetup{%
  pdftitle={Local existence and uniqueness in the largest critical space for a surface growth model},
  pdfauthor={D. Bloemker, M. Romito},
  pdfkeywords={surface growth, critical space, uniqueness, regularity},
  pdfcreator={M. Romito}
}
\begin{document}
  \title[Largest critical space for surface growth]{Local existence and uniqueness in the largest critical space for a surface growth model}
  \author[D. Bl\"omker]{Dirk Bl\"omker}
    \address{Institut f\"ur Mathematik\\ Universit\"at Augsburg\\ D-86135 Augsburg, Germany}
    \email{dirk.bloemker@math.uni-augsburg.de}
    \urladdr{\url{http://www.math.uni-augsburg.de/ana/bloemker.html}}
  \author[M. Romito]{Marco Romito}
    \address{Dipartimento di Matematica, Universit\`a di Firenze\\ Viale Morgagni 67/a\\ I-50134 Firenze, Italia}
    \email{romito@math.unifi.it}
    \urladdr{\url{http://www.math.unifi.it/users/romito}}
  \thanks{This work has been partially supported by the GNAMPA project 
    \emph{Studio delle singolarit\`a di alcune equazioni legate a modelli idrodinamici}.
    Part of the work was done at the Newton institute for Mathematical Sciences
    in Cambridge (UK), whose support is gratefully acknowledged, during the
    program "Stochastic partial differential equations". The authors would like
    to thank Herbert Koch for pointing out the approach to the problem.}
  \subjclass[2000]{35B33, 35B45, 35B65, 35K55, 35Qxx, 60H15}
  \keywords{surface growth, critical space, uniqueness, regularity}
  \date{Cambridge, March 22, 2010}
  \begin{abstract}
    We show the existence and uniqueness of solutions 
(either local or global for small data) for an equation
    arising in different aspects of surface growth. Following the work of Koch
    and Tataru we consider spaces critical with respect to scaling and we prove
    our results in the largest possible critical 
    space such that weak solutions are defined. 
    The uniqueness of global weak solutions remains unfortunately open, unless the
    initial conditions are sufficiently small.
  \end{abstract}
\maketitle

\section{Introduction}


The analysis of mathematical models for the study of surface growth has attracted
a lot of attention in recent years, one can see for example the reviews in
\cite{BaSt95, HHZh95} and numerous recent publications. See for example 
\cite{Rai-a,Rai-b,Rai-c,sput-1,sput-2,sput-3,FrVe:06}, which we comment in detail later.

In this article we consider a model arising in the growth of amorphous surfaces
which is described by the following partial differential equation,
\begin{equation}\label{e:sg}
\partial_t h + \Delta^2 h + \Delta|\nabla h|^2 = 0.
\end{equation}
on the whole $\R^d$ or with periodic boundary conditions.
The function $h(t,\cdot)$ models a height profile
at time $t>0$, so $d=1$ and $d=2$ are 
the physically relevant dimensions. 
In view of this and of Proposition \ref{p:icok} 
we will restrict the analysis to the case $d\le 3$ 
throughout this paper (although most of the computation 
holds without restrictions on the dimension).

Equation \eqref{e:sg}, which is sometimes referred to as a conservative version 
of the Kuramoto-Sivashinsky equation, arises also in several other models for
surface growth. The two--dimensional version was suggested in \cite{Rai-a,Rai-b,Rai-c} 
as a phenomenological model for the growth of an amorphous surface
(Zr$_{65}$Al$_{7,5}$Cu$_{27,5}$) and more recently as a model in surface
erosion using ion-beam sputtering \cite{sput-1,sput-2,sput-3}.
The one-dimensional equation appeared as  a model for the boundaries of
terraces in the epitaxy of Silicon \cite{FrVe:06}.

For simplicity of presentation we consider the rescaled version~\eqref{e:sg} 
with a--dimensional length-scales. Furthermore, we have ignored lower order
terms  like the Kuramoto-Sivashinsky term $-|\nabla h|^2$ or a linear
instability given by $+\Delta h$. These terms can easily be incorporated in
the result.

In the physical literature equation \eqref{e:sg} is usually subject 
to space-time white noise, which we also have neglected for 
simplicity of presentation. Indeed, using the standard method of 
looking at the difference between $h$ and the stochastic convolution,
the stochastic PDE can be transformed in a random PDE.
If the stochastic convolution is sufficiently regular,
then for each instance of chance the path-wise solvability  
for the stochastic PDE  is completely analogous to the results 
presented here and one only needs to consider additional lower order terms.
This will be done with more details later in Section~\ref{sec:SPDE}.

A crucial open problem for equation~\eqref{e:sg}, 
is the fact that the uniqueness of global solutions is not known.
We remark that numerical experiments do not report any problems of blow up,
see Hoppe and Nash~\cite{HoNa04,HoNa02}, or the previously stated physics
literature. Numerical experiments from Bl\"omker, Gugg and Raible~\cite{BlGuRa02} 
furthermore indicate a fast convergence of spectral Galerkin methods for
averaged surface roughness for the stochastic PDE.

The existence of global weak solutions in dimension $d=1$ on bounded domains
has been studied in \cite{BlGuRa02} (see also the references therein), based
on spectral Galerkin methods. The crucial estimates are energy-type inequalities 
which allow for uniform bounds on the $L^2$-norm. The method has been significantly 
extended by Bl\"omker, Flandoli and Romito~\cite{BlFlRo09} 
in order to verify the existence of a solution that defines a Markov process.
Winkler and Stein \cite{StWi05} used Rothe's method to verify the existence
of a global weak solution, this result has been recently extended 
by Winkler~\cite{Wi:P} to the two--dimensional case, using energy type
estimates for $\int e^h\,dx$.  
 
The authors have showed in \cite{BloRom09} the uniqueness of local solutions
with initial values in the critical Hilbert space $H^{1/2}$ in the one
dimensional case. Local uniqueness of continuous solutions in $W^{1,4}$ for
the stochastic PDE in dimension $d=1,2$ can be found in \cite{BlGu04}.
A regularized problem with a cut-off in the nonlinearity in dimension $d=2$
has been studied in Hoppe, Linz and Litvinov~\cite{HoLiLi03}.

In this paper we study existence and uniqueness of solutions with initial
data in a space of BMO--type, which contains all previous spaces where analogous
results were proved. For periodic boundary conditions it allows for unique
local solutions with arbitrary initial data in $H^{1/2}$ or the space of
continuous functions $C^0$.   

A weak solution for~\eqref{e:sg} with initial condition $h_0\in L^1_\loc(\R^d)$
is any distribution $h$ on $\R^d$ with locally square integrable gradient
$\nabla h\in L^2_\loc([0,\infty)\times\R^d)$ such that for every smooth and
compactly supported test function $\phi\in C^\infty_c([0,\infty)\times\R^d)$,
\begin{multline}
\label{e:weak}
\int_0^\infty\int_{\R^d} h(t,x)\frac{\partial\phi}{\partial t}(t,x)\,dx\,dt
 - \int_0^\infty\int_{\R^d} h(t,x)\Delta^2\phi(t,x)\,dx\,dt +{}\\
 - \int_0^\infty\int_{\R^d} |\nabla h(t,x)|^2\Delta\phi(t,x)\,dx\,dt
 = -\int_{\R^d} h_0(x)\phi(0,x)\,dx.
\end{multline}
Note that $\nabla h\in L^2_\loc([0,\infty)\times\R^d)$ implies also
$h\in L^2_\loc([0,\infty)\times\R^d)$ (cf. Lemma~\ref{lem:disreg}) and
thus all terms in in~\eqref{e:weak} are well-defined,
not only in the sense of distributions.
Moreover, the solution is only defined up to constants.

Following the remarkable paper by Koch and Tataru \cite{KocTat01}, this article provides
a local existence and uniqueness result in the largest critical
 space, where the  above stated definition
of weak solutions makes sense. As the equation is translation
invariant (in space) and invariant with respect to the scaling
\begin{equation}\label{e:scaling}
  h(t, x) \longrightarrow h(\lambda^4 t, \lambda x),
\end{equation}
we consider the following scaling-aware invariant version of the $L^2_\loc$ space
for the gradient $\nabla h$,
\begin{equation}\label{e:spaceX0}
\|h\|_{\X^0}
 := \Bigl(\sup_{x\in\R^d,R>0}\Big\{\frac1{R^{d+2}}\int_0^{R^4}\int_{B_R(x)} |\nabla h|^2\,dy\,dt\Big\}\Bigr)^{\frac12}.
\end{equation}
The paper is organized as follows. In Section~\ref{sec:FS} we discuss the space 
defined by \eqref{e:spaceX0} and show an equivalent representation,
and its relation to BMO-type spaces. Some admissible initial conditions and
examples are discussed in Section~\ref{sec:exp}.

Based on Banach's fixed-point iteration-scheme, Section~\ref{sec:FPA} provides
the existence  and uniqueness results. Section~\ref{sec:SPDE} contains some
details on the extension of such results to the stochastically forced case.
We close the paper with Section \ref{sec:smooth}, where we show smoothness
of solutions. 
\section{Function spaces}\label{sec:FS}

Recall first the following result, an easy consequence of Poincar\'e's
inequality, which ensures that all integrals in~\eqref{e:weak} are well
defined.
\begin{lemma}\label{lem:disreg}
If $u$ is a distribution on $\R^d$ such that
$\nabla u\in L^2_\loc([0,\infty)\times\R^d)$, then
$u\in L^2_\loc([0,\infty)\times\R^d)$ and thus
$u\in L^2_\loc([0,\infty), H^1_\loc(\R^d))$.
\end{lemma}
We consider the linear space 
$\X^0$ of functions $h$ with  
 $|\nabla h| \in L^2_\loc((0,\infty)\times\R^d)$ 
and thus $ h \in L^2_\loc((0,\infty)\times\R^d)$
such that the quantity $\|h\|_{\X^0}$ is finite.

Furthermore, we define the linear space $\X$ of
functions such that the following norm is finite.
\begin{equation}\label{e:spaceX}
\|k\|_\X
 = \sup_{t>0}\bigl\{t^{\frac14}\|\nabla k(t)\|_\infty\bigr\}.
\end{equation}
A local in time version of these spaces can be defined for any $R>0$ by
\[
\begin{gathered}
  \|k\|_{\X_R^0}^2
    := \sup_{x\in\R^d,r\leq R}\Big\{\frac1{r^{d+2}}\int_0^{r^4}\int_{B_r(x)} |\nabla k(t,y)|^2\,dy\,dt
    \Big\}\;,\\
  \|k\|_{\X_R}
    := \sup_{t\leq R^4}\bigl(t^{\frac14}\|\nabla k(t)\|_\infty\bigr).
\end{gathered}
\]
for functions $k:[0,R^4]\times\R^d\to\R$.
Note that we always identify functions that differ only by a constant. This
is motivated by the fact that the equation is mass-conservative, if the total
mass $\int h\,dx$ is finite.
Furthermore, solutions are only defined up to additive constants.

In order to track the corresponding spaces for initial values, let
$A=\Delta^2$. Consider the Green's function $G: [0,\infty)\times\R^d\to \R$ 
associated to the operator $A$, where $G(t,x)$ has the
Fourier transform (w.r.t.~$x$) $\widehat G(t,\xi) = \e^{-t|\xi|^4}$.  
By scaling we obtain
\[
  G(t,x)
    = t^{-d/4} g(xt^{-1/4}),
      \quad\text{where}\quad 
  g(x)
    = G(1,x).
\]
The function $g$ is in the Schwartz class since $\widehat g(\xi) = \e^{-|\xi|^4}$.

Define the semigroup $\e^{-tA}$ by the convolution $\e^{-tA}k = G(t,\cdot)\star k$.
Define the space $\B^0$ of all functions $k:\R^d\to\R^d$  
such that  the \emph{bi-caloric} extension $\e^{-tA}k$ is in $\X^0$,
endowed with the semi-norm
\[
  \|k\|_{\B^0}
    := \|\e^{-tA}k\|_{\X^0},
\]
and the space $\B$ of all functions $k:\R^d\to\R^d$ such that
\[
  \|k\|_\B
    := \|\e^{-tA}k\|_\X
\]
is finite, endowed with the semi-norm $\|\cdot\|_{\B}$.
Define similarly the local versions $\B^0_R$ and $\B_R$ of these spaces.

In contrast to the case of Navier--Stokes in dimension three~\cite{KocTat01},
the spaces $\B$ and $\B^0$ (as well as their local counterparts) turn out to
be equivalent, as shown by the proposition stated below. This makes the
analysis simpler than in the Navier--Stokes case. This proposition is the only 
reason, why we restrict to dimension $d\le3$, as we rely for simplicity on the
equivalence of $\B$ and $\B0$. 
\begin{proposition}\label{p:icok}
Assume $d\leq 3$. Then there are constants $c_1$, $c_2>0$ such that
\begin{equation}\label{e:icokglobal}
c_1\|\cdot\|_{\B}\leq\|\cdot\|_{\B^0}\leq c_2\|\cdot\|_{\B}
\end{equation}
and
\begin{equation}\label{e:icoklocal}
c_1\|\cdot\|_{\B_R}\leq\|\cdot\|_{\B^0_R}\leq c_2\|\cdot\|_{\B_R},
\end{equation}
for every $R>0$.

Moreover, for every $d\geq1$, there exists $c_3>0$ such that
\[
  \|k\|_{\B^0} \leq c_3\|k\|_{BMO(\R^d)},
\]
so $BMO(\R^d)\subset\B^0$ and in particular $L^\infty(\R^d)\subset\B^0$.
\end{proposition}
For a definition of the space $BMO(\R^d)$ of functions of bounded mean oscillation 
and its properties, we refer to
Stein \cite{Ste93}. Here we only use an equivalent norm on $BMO(\R^d)$
given by the Carleson measure characterization (see \eqref{e:BMO}).
\begin{proof}
We start by proving~\eqref{e:icokglobal}. The inequality on the right holds
in any dimension $d\geq1$ since it is straightforward to check that there is
$c>0$ such that $\|\cdot\|_{\X^0}\leq c\|\cdot\|_\X$. For the inequality
on the left, we need to show that for $k\in\B^0$, $x\in\R^d$, $t>0$,
\[
  |t^{\frac14}\nabla(\e^{-tA}k)(x)|
     = t^{\frac14}\Bigl|\nabla\int_{\R^d} G(t,x-y)k(y)\,dy\Bigr|
    \leq c\|k\|_{\B^0}.
\]
By scaling and translations invariance, it is sufficient to show the statement
for $t=1$ and $x=0$. Since
\[
  \e^{-A} = \int_0^1\e^{-(1-s)A}\e^{-sA}\,ds,
\]
it follows by the Cauchy Schwartz inequality that
\[
\begin{aligned}
  |\nabla(\e^{-A}k)(0)|
    & = \Bigl|\int_0^1{(1-s)^{-\frac{d}4}}
          \int_{\R^d} g\bigl(y(1-s)^{-\frac14}\bigr)
	  \nabla(\e^{-sA}k)(y)\,dy\,ds\Bigr|\\
    &\leq \sum_{n\in\Z^d}\Bigl|\int_0^1(1-s)^{-\frac{d}4}
            \int_{B_n} g\bigl(y(1-s)^{-\frac14}\bigr)\nabla(\e^{-sA}k)(y)\,dy\,ds\Bigr|\\
    &\leq \sum_{n\in\Z^d} \Bigl(\int_0^1(1-s)^{-\frac{d}{2}}
	      \int_{B_n} g\bigl(y(1-s)^{-\frac14}\bigr)^2\,dy\,ds\Bigr)^{\frac12}\\
    &         \hspace*{4cm} \times
	      \Bigl(\int_0^1\int_{B_n}|\nabla(\e^{-sA}k)(y)|^2\,dy\,ds\Bigr)^{\frac12}\\
    &\leq c\|k\|_{\B^0} \sum_{n\in\Z^d} \Bigl(\int_0^1s^{-\frac{d}{2}}\int_{B_n} g\bigl(ys^{-\frac14}\bigr)^2\,dy\,ds\Bigr)^{\frac12},
\end{aligned}
\]
where $B_n$ are the balls of centre $2d^{-1/2}n$ and radius $1$ (so that their
union covers $\R^d$). By a change of variables,
\[
  I_n
    := \int_0^1s^{-\frac{d}{2}}\int_{B_n} g\bigl(ys^{-\frac14}\bigr)^2\,dy\,ds
    = \int_0^1s^{-\frac{d}{4}}\int_{s^{-1/4}B_n} g(z)^2\,dz\,ds.
\]
First, $|I_n|\le C$ for all $n\in\Z^d$, as $d\le3$ and $g\in L^2(\R^d)$.
Note that $d\le3$ is necessary, as for $0\in B_n$ we have 
$\int_{s^{-1/4}B_n} g(z)^2\,dz \uparrow \|g\|^2_{L^2}$ for $s\downarrow 0$.

For the convergence of the series consider for $s\in(0,1)$ 
and $0\not \in B_n$ (i.e.~$2|n|>\sqrt{d}$) that 
\[
\begin{aligned}
 \int_{s^{-1/4}B_n} g(z)^2\,dz
   &\leq \int_{\R^d}|g(z)|\,dz \cdot \sup\left\{|g(z)| \ : \ |z|\in s^{-1/4}B_n \right\} \\
   &\leq C\sup\left\{|g(z)| \ : \ |z|> 2|n|d^{-1/2}-1 \right\},
\end{aligned}
\]
which can be bounded by a summable term, since  $g$ is in the Schwartz class.
The inequality~\eqref{e:icoklocal} for the local spaces proceeds similarly.

Let $\phi$ be in the Schwartz class, with $\widehat\phi>0$,
and set $\phi_t(x)=t^{-d}\phi(\tfrac{x}{t})$. By the Carleson measure characterization
of $BMO$ (see Theorem 3, Section 4.3 of Stein~\cite{Ste93}), we have that (up
to a constant)
\begin{equation} \label{e:BMO}
  \|k\|_{BMO}^2
    = \sup_{x\in\R^d,R>0}\Big\{\frac{1}{R^d}\int_0^R\int_{B_R(x)}\frac1t  |\bigl((\nabla\phi)_t\star k\bigr)(y)|^2\,dy\,dt\Big\}\;.
\end{equation}
Note that this is an equivalent norm to the standard definition.
We could relax the conditions on $\phi$,
but we are going to use $\phi=g$ which satisfies 
the stronger condition $\widehat g>0$.
On the other hand the definition of $\B^0$ given above can be restated (up to
a constant) as
\[
  \|k\|_{\B^0}^2
    = \sup_{x\in\R^d,R>0}\Big\{\frac{1}{R^{d+2}}\int_0^R\int_{B_R(x)} t|\bigl((\nabla\phi)_t\star k\bigr)(y)|^2\,dy\,dt\Big\}\;,
\]
with $\phi = g$, and so $\|k\|_{\B^0} \leq c_3\|k\|_{BMO(\R^D)}$, as $t^2 \le R^2$.
\end{proof}
\section{Examples}\label{sec:exp}

In view of Theorem~\ref{t:koch} we wish to discuss for which initial conditions
it is possible to find $R$ such that the initial condition is small in the
$\B_R$ norm. To this aim define
\[
  \vmo = \{k:\R^d\to\R: \|k\|_{\B_R}\to0\text{ as }R\downarrow0\}.
\]
We later see in Theorem~\ref{t:koch} 
that functions in $\vmo$ correspond to initial conditions,
where it is possible to solve the equation~\eqref{e:sg} locally for 
a small time interval.  

The next lemma shows that the whole $L^\infty(\R^d)$, although it is contained in
$\B^0$, is not contained in $\vmo$. We will see later that 
this implies that our method of proof 
fails to provide local uniqueness 
of solutions for some initial 
conditions in $L^\infty(\R^d)$, although $L^\infty(\R^d)\subset \B_0$.
\begin{lemma}\label{l:vmo}
The following statements hold,
\begin{itemize}
\item there are functions in $L^\infty(\R^d)$ not belonging to $\vmo$,
\item if $k:\R^d\to\R$ is bounded and uniformly continuous, then $k\in\vmo$,
\item if $k:\R^d\to\R$ has bounded gradient on $\R^d$, then $k\in\vmo$.
\end{itemize}
\end{lemma}
\begin{proof}
We prove the first statement. Since the Green's function tensorises, it is
enough to find a counterexample in dimension $d=1$. Let $k(x)=\uno_{[-1,1]}(x)$,
then
\[
\begin{aligned}
  t^{\frac14}|(\e^{-tA}k)_x(x)|
    & = \Bigl|\frac{1}{t^\frac14}\int_\R g'(\tfrac{x-y}{t^{1/4}})k(y)\,dy\Bigr|
      = |g(\tfrac{x+1}{t^{1/4}}) - g(\tfrac{x-1}{t^{1/4}})|,
\end{aligned}
\]
hence (choosing $x=-1$ and $t=R$), since $g(x)\to0$ as $x\to\infty$,
\[
  \sup_{t\leq R, x\in\R}\left\{ t^{\frac14}|(\e^{-tA}k)_x(x)|\right\}
    \geq |g(0) - g(\tfrac{-2}{R^{1/4}})|
    \longrightarrow g(0)>0.
\]
Assume now that $k$ is bounded and uniformly continuous and fix $\epsilon>0$.
By uniform continuity there is $\delta>0$ such that $|k(x)-k(y)|\leq\epsilon$
for all $x,y\in\R^d$ with $|x-y|\leq\delta$. Since the integral of $\nabla g$
is zero,
\[
\begin{aligned}
  t^\frac14|\nabla(\e^{tA}k)(x)|
    & = \Bigl|\int_{\R^d} \nabla g(z)\bigl(k(x-zt^{\frac14}) - k(x)\bigr)\,dz\Bigr|\\
    & = \Bigl|\int_ {t^{\frac14}|z|\geq\delta} \nabla g(z) \bigl(k(x-zt^{\frac14}) - k(x)\bigr)\,dz\Bigr|\\
    &\hspace*{3cm} + \Bigl|\int_ {t^{\frac14}|z|\leq\delta}\nabla g(z)\bigl(k(x-zt^{\frac14}) - k(x)\bigr)\,dz\Bigr|\\
    &\leq 2\|k\|_\infty\int_{t^{\frac14}|z|\geq\delta} |\nabla g(z)|\,dz
        + \epsilon \|\nabla g\|_{L^1(\R^d)},
\end{aligned}
\]
hence
$\limsup_{R\to 0} \|k\|_{\B_R}\leq\epsilon\|\nabla g\|_{L^1}$ and as
$\epsilon\downarrow0$, the claim follows.

Finally, let $k:\R^d\to\R$ be such that $\|\nabla k\|_{L^\infty}<\infty$, then
\[
  t^\frac14|\nabla(\e^{tA}k)(x)|
    = t^{-\frac{d-1}{4}} \Bigl|\int_{\R^d} g(\tfrac{x-y}{t^{1/4}}) \nabla k(y)\,dy\Bigr|
    \leq t^\frac14 \|g\|_{L^1} \|\nabla k\|_{L^\infty},
\]
and hence $\|k\|_{\B_R} = R^\frac14 \|g\|_{L^1} \|\nabla k\|_{L^\infty}$.
\end{proof}
\begin{example}
For $d=1$ consider $k(x) = \log|x|$. This function has paramount
importance since it is a stationary solution for problem~\eqref{e:sg}
(see~\cite{BloRom09}). It is an interesting fact that $k$ is neither 
a weak nor a mild solution (for instance due to Theorem~\ref{t:smooth} which
ensures smoothness of solutions according to Definition~\ref{def:mild}). Here
we will show that $k\in\B$ but $k\not\in\vmo$.
\end{example}
Indeed, consider first,
\[
\begin{aligned}
  t^{\frac14}\bigl|(\e^{-tA} k)_x(x)\bigr| 
    & = \Bigl|\int_\R g'(z) \log|x- t^{\frac14}z|\,dz \Bigr|\\
     & = \Bigl|\int_\R g'(z) (\log(t^{\frac14}) + \log| t^{-\frac14}x- z|)\,dz \Bigr|  \\
    & = \Bigl|\int_\R g'(z) \log| t^{-\frac14}x- z|\,dz \Bigr| \;,
\end{aligned}
\]
where we used that the integral over $g'$ is zero. Now substitute
$\tilde{x}=t^{-\frac14}x$ to obtain
\[
  \sup_{x\in\R}  \Big\{t^{\frac14}|(\partial_x\e^{-tA} k) (x)| \Big\}
    = \sup_{x\in\R} \Big\{ \Bigl| \int_\R g'(z) \log|x- z|\,dz \Bigr| \Big\}
    = \|\partial_x\e^{-A}k\|_\infty \;.
\]
So it is easy to see that $\|k\|_{\B}$ is finite,
but $\|k\|_{\B_R}$ is independent of $R$ and does not converge to $0$.

\begin{example}
Consider $d=1$ and, for $\alpha>0$, $k_\alpha(x)=|x|^\alpha$.
Then $k_\alpha\not\in\B$, but $\|k_\alpha\|_{\B_R}\to0$. So $\vmo$ may
contain certain unbounded functions, which are not in $\B$. 
\end{example}
As in the previous example,
\[
  t^{\frac14}|(\partial_x\e^{-tA} k_\alpha)(x)| 
    = \Bigl|\int_\R g'(z) |x- t^{\frac14}z|^\alpha\,dz \Bigr|
    =  t^{\frac\alpha4} \Big| \int_\R g'(z) | t^{-\frac14}x- z|^\alpha\,dz \Big|
\]
Thus $\|t^{\frac14}\partial_x(\e^{-tA} k_\alpha)\|_\infty = t^{\frac\alpha4} \|\partial_x(\e^{-A}k_\alpha)\|_\infty$,
hence $\|k_\alpha\|_{\B_R}\to 0$ for $R\to0$ but $\|k_\alpha \|_{\B}=\infty$.

Next lemma, together with the main Theorem~\ref{t:koch}, shows that
problem~\eqref{e:sg} has locally  a unique solution 
for any $\dot H^{d/2}(\R^d)$ initial conditions. This
recovers and extends a result proved in dimension $d=1$ in~\cite{BloRom09}.
\begin{lemma}
The homogeneous space $\dot H^{d/2}(\R^d)$ is contained in $\vmo$, where
\[
  \dot H^{d/2}(\R^d)
    = \Bigl\{k:\R^d\to\R: \|k\|_{\dot H^{d/2}}:=\int_{\R^d}|\xi|^d|\widehat k(\xi)|^2\,d\xi<\infty\Bigr\}
\]
and $\widehat k$ denotes Fourier transform of $k$.
\end{lemma}
\begin{proof}
If $t>0$ and $x\in\R^d$, by using the properties of Fourier transform and
convolution,
\[
\begin{aligned}
  t^\frac14|\nabla(\e^{tA}k)(x)|
    & = \frac{1}{t^\frac{d}4}\Bigl|\int_{\R^d}(\nabla g)\bigl((x-y)t^{-\frac14}\bigr)k(y)\,dy\Bigr|\\
    & = \frac{1}{t^\frac{d}4}\Bigl|\int_{\R^d}\int_{\R^d}(\nabla g)\bigl((x-y)t^{-\frac14}\bigr)\widehat k(\xi)\e^{\im\xi\cdot y}\,dy\,d\xi\Bigr|\\
    & = \Bigl|\int_{\R^d}\int_{\R^d} \nabla g(y)\e^{-\im\xi z t^\frac14} \widehat k(\xi)\e^{\im\xi\cdot x}\,dy\,d\xi\Bigr|\\
    &\leq t^\frac14 \int_{\R^d} |\xi|\,|\widehat k(\xi)|\e^{-|\xi|^4t}\,d\xi.
\end{aligned}
\]
Given $a>0$, split the integral in the last line of formula above in two pieces
\term{l} and \term{h}, corresponding to the domains of integration
$\{|\xi|\leq a\}$ and $\{|\xi|> a\}$ respectively. We estimate \term{l}
using the Cauchy--Schwarz inequality,
\[
\term{l}
  \leq t^\frac14 \int_{|\xi|\leq a} |\xi|^{1-\frac{d}2}\bigl(|\xi|^\frac{d}2|\widehat k(\xi)|\bigr)\,d\xi
  \leq t^\frac14 \|k\|_{\dot H^{d/2}}\Bigl(\int_{|\xi|\leq a} |\xi|^{2-d}\,d\xi\Bigr)^\frac12
  \leq c a t^\frac14 \|k\|_{\dot H^{d/2}},
\]
while by a change of variables and Cauchy--Schwarz' inequality again,
\[
\term{h}
  \leq t^\frac14 \Bigl(\int_{|\xi|\geq a} |\xi|^d|\widehat k(\xi)|^2\,d\xi\Bigr)^\frac12
                 \Bigl(\int_{\R^d} |\xi|^{2-d}\e^{-2t|\xi|^4}\,d\xi\Bigr)^\frac12
  \leq c \Bigl(\int_{|\xi|\geq a} |\xi|^d|\widehat k(\xi)|^2\,d\xi\Bigr)^\frac12.
\]
In conclusion
\[
  \|k\|_{\B_R}
    \leq c a R \|k\|_{\dot H^{d/2}}
       + c \Bigl(\int_{|\xi|\geq a} |\xi|^d|\widehat k(\xi)|^2\,d\xi\Bigr)^\frac12,
\]
so we see that $\limsup_{R\to0}\|k\|_{\B_R}$ is bounded by a quantity which
converges to $0$ as $a\uparrow\infty$.
\end{proof}
\section{The fixed point argument}\label{sec:FPA}

Define the map
\[
  \mapV(h,k)(t)
    = \int_0^t \Delta\bigl(\e^{-(t-s)A}\nabla h(s)\cdot\nabla k(s)\bigr)\,ds
\]
and set
\begin{equation}\label{e:map}
\map(h)(t) = \e^{-tA} h_0 - \mapV(h,h)(t).
\end{equation}
We will use the following concept of a mild solution,
which is given as a solution of the variation of constants formula in \eqref{e:mild}.
\begin{definition}
\label{def:mild}
We say that $h\in\X$ solves~\eqref{e:sg} with initial condition $h_0\in\B$, 
if for all $t>0$
\begin{equation}\label{e:mild}
h(t)=\e^{-tA} h_0 - \mapV(h,h)(t)\;.
\end{equation}
We call  $h\in \X_R$ a local solution, if \eqref{e:mild} holds only for $t\in [0,R^4]$.
\end{definition}
The following Lemma is crucial for the proof of uniqueness and existence.
It verifies that the nonlinear part is locally Lipschitz.
\begin{lemma}\label{l:map}
The map $\mapV$ is bi-linear continuous from $\X\times\X$ to $\X$ and from
$\X_R\times\X_R$ to $\X_R$, for all $R>0$.
\end{lemma}
\begin{proof}
The bilinearity is obvious. For the boundedness 
let $x\in\R^d$ and $t>0$, then
\[
\begin{aligned}
  |\nabla\mapV(h,k)(t,x)|
    & = \Bigl|\int_0^t\int_{\R^d}\nabla\Delta G(t-s,x-y) \nabla h(s,y) \nabla k(s,y)\,dy\,ds\Bigr|\\
    & = \Bigl|\int_0^t\frac1{(t-s)^{\frac{d+3}{4}}}\int_{\R^d} (\nabla\Delta g)\Bigl(\frac{x-y}{(t-s)^{1/4}}\Bigr) \nabla h(s,y) \nabla k(s,y)\,dy\,ds\Bigr|\\
    &\leq \|h\|_\X \|k\|_\X \int_0^t\frac1{(t-s)^{\frac{d+3}{4}}\sqrt{s}}\int_{\R^d} \Bigl|(\nabla\Delta g)\Bigl(\frac{x-y}{(t-s)^{1/4}}\Bigr)\Bigr|\,dy\,ds\\
    &\leq \|h\|_\X \|k\|_\X \|g\|_{W^{3,1}(\R^d)} \int_0^t \frac1{(t-s)^{3/4}\sqrt{s}}\,ds\\
    &\leq t^{-1/4}  B(\tfrac12,\tfrac14) \|h\|_\X \|k\|_\X \|g\|_{W^{3,1}(\R^d)}\\
    & =   c_4 t^{-1/4}\|h\|_\X \|k\|_\X,
\end{aligned}
\]
where $B$ is the Beta function. The corresponding inequality for the local
space $\X_R$ proceeds similarly.
\end{proof}
Using the previous Lemma, we can now state and prove our main result.
The first part states global existence of unique solutions,
while the second part is about local existence of solutions.
\begin{theorem}\label{t:koch}
There is $\delta>0$ such that if $\|h_0\|_{\B^0}\leq\delta$, then there exists 
a unique (global) solution in $\X$ of \eqref{e:sg} with initial condition $h_0$.

Moreover, if $\|h_0\|_{\B^0_R}\leq\delta$, then there is a unique local solution
in $\X_R$ of \eqref{e:sg} on $[0,R^4]$  with initial condition $h_0$.

Finally, if $h_0$ is periodic and small in $\B^0_R$ for some $R>0$ (or it
is small in $\B^0$), then the solution is also periodic.
\end{theorem}
In particular, $\|h_0\|_{\B^0_R}\leq\delta$ is true for a suitable value
of $R$ for all $h_0\in\vmo$.
\begin{proof}
We prove the first statement by a fixed point iteration argument. Let $c_4$ be the
constant defined in the proof of Lemma~\ref{l:map} and choose $\delta>0$,
$K>0$ such that
\[
  1-\frac{4 c_4\delta}{c_1}>0,
    \qquad
  \frac{1}{2c_4}\Bigl(1-\sqrt{1-\frac{4 c_4\delta}{c_1}}\Bigr)\leq K<\frac{1}{2c_4}.
\]
Define
\begin{equation}\label{e:iter}
  H_0 = 0,
    \qquad
  H_{n+1} = \map(H_n) = \e^{-tA}h_0 - \mapV(H_n,H_n),
\end{equation}
then $\|H_1\|_\X\leq\tfrac{\delta}{c_1}$ and it is easy to check by induction
(and by the choice of $\delta$ and $K$) that $\|H_n\|_\X\leq K$ for all $n$.
Then
\[
  \|H_{n+1}-H_n\|_\X
    =    \|\mapV(H_n,H_n) - \mapV(H_{n-1},H_{n-1})\|_\X
    \leq 2c_4 K \|H_n-H_{n-1}\|_\X
\]
and so $(H_n)_{n\in\N}$ is convergent in $\X$ to a fixed point of $\map$.

The same proof works for local spaces, since both constants $c_1$ and $c_4$ do
not depend on $R$. Finally, if $h_0$ is periodic, the statement follows by
translation invariance and uniqueness.
\end{proof}
\begin{remark}[Forward self--similar solutions]
The theorem above allows to show the existence of self--similar solutions,
namely solutions invariant for the scaling~\eqref{e:scaling}. Indeed, assume
to have $h_0\in\B^0$ (or in a local space) such that $h_0(\lambda x) = h_0(x)$
for all $\lambda>0$, then it is easy to verify that $H_1$ is invariant for
the scaling~\eqref{e:scaling} and that $\mapV(h,h)$ is also invariant if so
is $h$. In conclusion the whole sequence $(H_n)_{n\in\N}$ defined in~\eqref{e:iter}
is invariant, as well as its limit.

Given a (forward) self--similar solution $h$, one can write $h(t,x) = \psi(x/t^{1/4})$,
where $\psi(x) = h(1,x)$ solves the equation
\[
\Delta^2\psi + \Delta|\nabla\psi|^2 - \frac14 x\cdot\nabla\psi = 0.
\]
The simplest case corresponds to $d=1$, where the only admissible initial
conditions are all functions $h_0$ constants on $(-\infty,0)$ and on $(0,\infty)$
(possibly with different values on the two half--lines), with
$\|h_0\|_{\B_R} = |h_0(1)-h_0(-1)|\,\|g\|_{L^\infty}$.

Backward self--similar solutions might provide examples of solutions with
blow--up. Due to the scaling of the problem, the quantity blowing up is
related to the derivative of the solution. We do not know if backward
self--similar solutions exist (notice that backward self--similar solutions
do not exist for the Navier--Stokes equations, see~\cite{NecRuzSve96}).
\end{remark}
\section{The stochastic problem}\label{sec:SPDE}

In this section we give a short outline of the proof 
of local existence for the stochastic PDE, 
without many details on probability theory.
For details we refer to \cite{DPZa:92,Ch:07,Liu:06}.
Consider  
\begin{equation}\label{e:spde}
\partial_t h + \Delta^2 h + \Delta|\nabla h|^2 = \partial_t W\;,
\end{equation}
where $ \partial_t W$ is the generalized derivative of a 
Hilbert-space value Wiener process.
Define the corresponding Ornstein-Uhlenbeck process 
for $t>0$ as the following It\^o-integral
\begin{equation}\label{e:mildspde}
Z(t)=\int_0^t e^{-(t-s)A}dW\;.
\end{equation}
Note that $Z$ solves $\partial_t Z + \Delta^2 Z= \partial_t W$ with $Z(0)=0$.
The  mild solution  of \eqref{e:spde} 
is analogous to Definition \ref{def:mild} given by a solution of
$$ h(t) = \e^{-tA} h_0 - \mapV(h,h)(t) + Z(t)
$$
Now the main problem in the stochastic setting is to determine the regularity
of $Z$. Once we know this, we can solve the equation using Banach's fixed
point argument, as in Theorem~\ref{t:koch}. 
Moreover $v=h-Z$ solves the following random PDE
\begin{equation}
\label{e:RPDE}
 \partial_t v+ \Delta^2 v + \Delta|\nabla v|^2 
= - \Delta|\nabla Z|^2 - 2\Delta|\nabla v\cdot\nabla Z|^2,
\qquad v(0)=h_0,
\end{equation}
which only contains lower order terms that do not change the proofs, once
$Z$ is sufficiently regular.

In the case of bounded intervals (i.~e.~$d=1$) with periodic boundary
conditions and space--time white noise the stochastic convolution $Z$
and its derivative $\partial_x Z$ are continuous in both space and time,
which can be verified using the methods in \cite{DPZa:92}. 
See for example \cite{DPDe:96}.
This implies that almost surely $\|Z\|_{\X_R}\to 0$ for $R\to 0$,
and we can solve the stochastic PDE \eqref{e:mildspde} (or the random PDE \eqref{e:RPDE}) 
uniquely in $\X_R$, for some small (random) $R>0$
if the initial condition $h_0$ is such that the PDE~\eqref{e:sg} 
has a unique local solution.

An interesting question appears in the case of periodic 
boundary conditions and $d=2$,
as for space-time white noise the convolution $Z$ 
just fails to be differentiable in space.
Nevertheless, $Z$ will be differentiable,
if we consider slightly more regular noise.

For stochastic PDEs on unbounded domains one can use the formulation of
Walsh~\cite{Wal86}, although one has to consider that for space--time white
noise the stochastic convolution $Z(t,x)$ is unbounded for $|x|\to\infty.$ 
\section{Smoothness of solutions}\label{sec:smooth}

Following the same methods of \cite{GerPavSta07}, we show that solutions in
$\X$ (or $\X_R$) are smooth. Define for $m\geq1$,
\[
  \|k\|_{\X,m}
    := \sup_{t>0} \Big\{t^{\frac{m+1}{4}}\!\!\sum_{|\alpha|=m+1}\!\! \|D^\alpha k\|_\infty\Big\}
\]
and denote by $\|\cdot\|_{\X_R,m}$ the corresponding local version, where
for $\alpha=(\alpha_1,\dots,\alpha_d)$ we used
$D^\alpha = \partial_{x_1}^{\alpha_1}\dots\partial_{x_d}^{\alpha_d}$ and
$|\alpha|=\alpha_1+\dots+\alpha_d$.

Let $\X^m$ be the space
\[
  \X^m
    = \{k:\R^d\to\R : \|k\|_{\X^m} := \max_{0\leq j\leq m} \|k\|_{\X,j}<\infty\},
\]
and denote by $\X_R^m$ the corresponding local version. For simplicity of
notations we understand that $\|\cdot\|_{\X,0} = \|\cdot\|_\X$
and for $R=\infty$ that $\X_\infty^m=\X^m$. 
The main theorem of this section is the following
result on smoothness in space. Smoothness in time then follows 
from the PDE by a standard bootstrapping argument. 
\begin{theorem}\label{t:smooth}
Let $h$ be a solution of~\eqref{e:sg} in $\X_R$, with $0<R\leq\infty$.
Then $h(t)\in C_b^\infty(\R^d)$ for all $t\in(0,R)$.
\end{theorem}
\begin{proof}
If the initial condition is small enough in $\B_R^0$, the statement follows
from Proposition~\ref{p:smooth} below. In the general case we notice that
if $h\in\X_R$, then $\nabla h(t)$ is bounded for all $t\in(0,R)$, therefore
$h(t)\in\vmo$, by Lemma~\ref{l:vmo}. The conclusion then follows again from
Proposition~\ref{p:smooth}.
\end{proof}
In order to complete the proof of the above theorem, we need the following
proposition, which gives also a better estimate of the solution near $t=0$
if the initial condition is small enough.
\begin{proposition}\label{p:smooth}
There exists $\delta>0$ such that if $\|h_0\|_{\B^0}<\delta$, then the solution
to \eqref{e:sg} granted by Theorem~\ref{t:koch} is in $\X^m$ for all $m\geq1$.

If $R>0$ and $\|h_0\|_{\B^0_R}<\delta$, then the solution to \eqref{e:sg}
granted by Theorem~\ref{t:koch} is in $\X_R^m$ for all $m\geq1$.
\end{proposition}
We start by giving a slight generalization of~\eqref{e:icokglobal}
and~\eqref{e:icoklocal}.
\begin{lemma}\label{l:ichigh}
Let $0<R\leq\infty$ and $k\in\B_R^0$, then for every $m\geq0$,
\begin{equation}\label{e:highic}
  \sup_{t\leq R}\Big\{t^{\frac{m+1}4}\sum_{|\alpha|=m+1} \bigl\|D^\alpha(\e^{-tA}k)\bigr\|_\infty\Big\}
    \leq c m^d (m+1)^{\frac{m+1}4}\|\nabla g\|_{L^1(\R^d)}^m\|k\|_{\B_R^0}.
\end{equation}
\end{lemma}
\begin{proof}
Since for $|\alpha|=m+1$,
\[
  D^\alpha(\e^{-tA}k)
    = \prod_{i=1}^{d}(\partial_{x_i}^{\alpha_i}\e^{-\frac{\alpha_i}{m+1}tA}) k,
\]
it is sufficient to show that the operator $\partial_{x_i}\e^{-tA}$ maps
$L^\infty(\R^d)$ into itself with operator norm
$\|\partial_{x_i}\e^{-tA}\|_{L^\infty\to L^\infty}\leq t^{-1/4}\|\nabla g\|_{L^1(\R^d)}$.
This is immediate since by a change of variables,
\[
  t^\frac14 |\partial_{x_i}(\e^{-tA}k)(x)|
      = t^{\frac{d}4}\Bigl|\int_{\R^d} (\partial_{x_i}g)\bigl(\tfrac{x-y}{t^{1/4}}\bigr)k(y)\,dy\Bigr|
   \leq \|\nabla g\|_{L^1(\R^d)}\|k\|_\infty.
\]
Finally, $\#(\{\alpha:|\alpha|=m+1\})=\binom{m+d}{d-1}\leq c m^d$.
\end{proof}
\begin{lemma}\label{l:nonlinhigh}
There is $c_5>0$ such that for $m\geq1$, $0<R\leq\infty$ and $h,k\in\X_R^m$,
\begin{equation}\label{e:highnonlin}
\begin{aligned}
  \|\mapV(h,k)\|_{\X_R,m}
    &\leq c_5 m^d(m+1)^{\frac{m+3}{2}} \|\nabla g\|_{L^1(\R^d)}^m \|g\|_{W^{3,1}(\R^d)} \|h\|_{\X_R} \|k\|_{\X_R}\\
    &\quad + c_5 \|h\|_{\X_R} \|k\|_{\X_{R,m}}
           + c_5 \|h\|_{\X_{R,m}} \|k\|_{\X_R}\\
    &\quad + c_5 m^d\sum_{j=1}^{m-1} \binom{m}{j}\|h\|_{\X_R,j} \|k\|_{\X_R,m-j}.
\end{aligned}
\end{equation}
\end{lemma}
\begin{proof}
Fix $m\geq1$, $0<R\leq\infty$, $t\leq R$ and $h, k\in\X_R^m$. 
Consider a value $\eps\in(0,1)$ which will be specified later, and let
$|\alpha|=m+1$. Since $|\alpha|\geq 1$, there is $i\leq d$ such that $\alpha_i\geq1$.
So assume without loss of generality that $a_1\geq1$ and let $\alpha'=\alpha-(1,0,\dots,0)$.
\[
\begin{aligned}
  D^\alpha\mapV(h,k)(t)
    & = \int_0^t D^\alpha\Delta\bigl(\e^{-(t-s)A}(\nabla h(s)\nabla k(s))\bigr)\,ds\\
    & = \int_0^{t(1-\eps)} D^\alpha\Delta\bigl(\e^{-(t-s)A}(\nabla h(s)\nabla k(s))\bigr)\,ds\\
     &\hspace*{2cm}+ \int_{t(1-\eps)}^t D^\alpha\Delta\bigl(\e^{-(t-s)A}(\nabla h(s)\nabla k(s))\bigr)\,ds\\
    & = \term{1} + \term{2}.
\end{aligned}
\]
For the term $\term{1}$ we use the factorization introduced in the previous
lemma and we proceed as in the proof of Lemma~\ref{l:map},
\[
\begin{aligned}
  |\,\term{1}\,|
    &\leq (m+1)^{\frac{m+3}4}\|\nabla g\|_{L^1}^m\|\partial_{x_1}\Delta g\|_{L^1}\|h\|_{\X_R}\|k\|_{\X_R}
            \int_0^{t(1-\eps)}s^{-\frac12}(t-s)^{-\frac{m+3}4}\,ds\\
    &\leq 2\sqrt{1-\eps}\bigl(\tfrac{m+1}{\eps}\bigr)^{\frac{m+3}{4}} t^{-\frac{m+1}{4}}
            \|\nabla g\|_{L^1}^m \|\partial_{x_1}\Delta g\|_{L^1} \|h\|_{\X_R}\|k\|_{\X_R}.
\end{aligned}
\]
For the second term we use Leibniz formula,
\[
  \term{2}
    = \sum_{\beta\leq\alpha'}\binom{\alpha'}{\beta}\int_{t(1-\eps)}^t
        \partial_{x_1}\Delta\e^{-(t-s)A}(D^\beta\nabla h)(D^{\alpha'-\beta}\nabla k)\,ds
\]
and, as in the proof of Lemma~\ref{l:map},
\[
\begin{aligned}
  \term{2}
    &\leq \sum_{\beta\leq\alpha'}\binom{\alpha'}{\beta} \|\partial_{x_1}\Delta g\|_{L^1}\|h\|_{\X_{R,|\beta|}}\|k\|_{\X_{R,m-|\beta|}}
            \int_{t(1-\eps)}^t\frac1{s^{\frac{m+2}{4}}(t-s)^{\frac34}}\,ds\\
    &\leq \frac{4\eps^{\frac14}}{(1-\eps)^{\frac{m+2}{4}}}t^{-\frac{m+1}{4}}\|\partial_{x_1}\Delta g\|_{L^1}
            \sum_{\beta\leq\alpha'}\binom{\alpha'}{\beta} \|h\|_{\X_{R,|\beta|}}\|k\|_{\X_{R,m-|\beta|}}
\end{aligned}
\]
If we set $\eps=\tfrac{1}{(m+d)^{4d}}$ the term $4\eps^{1/4}(1-\eps)^{-(m+2)/4}(m+d)^d$
is uniformly bounded in $m$ (we recall that the number of multi--indices
$\alpha$ such that $|\alpha|=m+1$ is bounded by $(m+d)^d$) and so by summing
up over $\alpha$ the estimates for $\term{1}$ and $\term{2}$ together show
the lemma.
\end{proof}
As in the proof of Theorem~\ref{t:koch}, define $H_0=0$ and
\[
  H_{n+1}(t)
    = \e^{-tA}h_0  - \mapV(H_n,H_n)(t).
\]
\begin{lemma}
There is $\delta'>0$ such that if $0<R\leq\infty$ and $\|h_0\|_{\B_R^0}<\delta$,
then for every $m\geq0$ there is $K_m>0$ such that
\[
  \|H_n\|_{\X_R,m}\leq K_m.
\]
\end{lemma}
Let us remark that with an explicit estimate of the constants 
$K_m$, and in particular their growth in terms of $m$, 
one could show that solutions are analytic in space.
For simplicity of presentation, we will not focus on this.
\begin{proof}
If $\|h_0\|_{\B_R^0}$ is small enough, the proof of Theorem~\ref{t:koch} shows
that there is $K_0$ such that $\|H_n\|_{\X_R}\leq K_0$.  By possibly taking
$\|h_0\|_{\B_R^0}$ smaller, we can assume that $\lambda=2c_5 K_0<1$, where
$c_5$ is given in Lemma~\ref{l:nonlinhigh}. We prove the statement by
induction: the case $m=0$ has been already proved. Set
$a_m=c_5 m^d (m+1)^{(m+3)/2}\|\nabla g\|_{L^1}^m\|g\|_{W^{3,1}}$ (this is
the coefficient appearing in the first line of formula~\eqref{e:highnonlin})
and $b_m=c m^d (m+1)^{(m+1)/4}\|\nabla g\|_{L^1}^m$ (this appears in formula
\eqref{e:highic}), then by Lemmas~\ref{l:ichigh} and \ref{l:nonlinhigh},
\[
\begin{aligned}
  \|H_{n+1}\|_{\X_R,m}
    &\leq \|H_1\|_{\X_R,m} + \|\mapV(H_n,H_n)\|_{\X_R,m}\\
    &\leq b_m K_0 + a_m K_0^2 + 2c_5K_0\|H_n\|_{\X_R,m}\\
    &\quad + c_5 m^d\sum_{j=1}^{m-1}\tbinom{m}{j}\|H_n\|_{\X_R,j}\|H_n\|_{\X_R,m-j}\\
    &\leq \Bigl(b_m K_0 + a_m K_0^2 + c_5 m^d \sum_{j=1}^{m-1}\tbinom{m}{j} K_j K_{m-j}\Bigr) + \lambda \|H_n\|_{\X_R,m},
\end{aligned}
\]
so that by recurrence and again Lemma~\ref{l:ichigh},
\[
\begin{aligned}
  \lefteqn{\|H_{n+1}\|_{\X_R,m}\leq}\\
\quad &\leq \Bigl(b_m K_0 + a_m K_0^2 + c_5 m^d \sum_{j=1}^{m-1}\tbinom{m}{j} K_j K_{m-j}\Bigr)(1+\dots+\lambda^{n-1}) + \lambda^n \|H_1\|_{\X_R,m}\\
      &\leq \frac1{1-\lambda}\Bigl(b_m K_0 + a_m K_0^2 + c_5 m^d \sum_{j=1}^{m-1}\tbinom{m}{j} K_j K_{m-j}\Bigr) + b_m K_0,
\end{aligned}
\]
and the last line in the formula above provides $K_m$.
\end{proof}
\begin{proof}[Proof of Proposition~\ref{p:smooth}]
Theorem~\ref{t:koch} ensures that if $\|h_0\|_{\B_R^0}$ is small enough, then
there is $\lambda = 2c_4 K_0<1$ (where the number $K_0$ is given by previous lemma)
such that $\|H_{n+1}-H_n\|_{\X_R}\leq c \lambda^n$. We prove by induction that
there are numbers $C_m>0$ and $\mu\in(0,1)$ such that
\[
  \|H_{n+1}-H_n\|_{\X_R,m}\leq C_m \mu^n,
    \qquad m\geq0,
\]
if $\|h_0\|_{\B_R^0}$ is small enough. Let $\lambda=2 c_5 K_0$ (where $c_5$ has been
introduced in Lemma~\ref{l:nonlinhigh}), assume $\lambda<1$ and let $\lambda<\mu<1$.
We have already verified that the inductive claim is true for $m=0$. Assume the
claim is true for $0$, \dots, $m-1$, then by Lemma~\ref{l:nonlinhigh} and the
inductive assumption,
\[
\begin{aligned}
  \|&H_{n+1} - H_n\|_{\X_R,m}\\
    &\leq \|\mapV(H_n,H_n-H_{n-1})\|_{\X_R,m} + \|\mapV(H_n-H_{n-1},H_{n-1})\|_{\X_R,m}\\
    &\leq    \bigl[a_m (\|H_n\|_{\X_R} + \|H_{n-1}\|_{\X_R}) + c_5(\|H_n\|_{\X_R,m} + \|H_{n-1}\|_{\X_R,m})\bigr]\|H_n-H_{n-1}\|_{\X_R}\\
    &\hspace*{1cm} + c_5(\|H_n\|_{\X_R}+\|H_{n-1}\|_{\X_R})\|H_n-H_{n-1}\|_{\X_R,m}\\
    &\hspace*{1cm} + c_5 m^d \sum_{j=1}^{m-1}\tbinom{m}{j} (\|H_n\|_{\X_R,j}+\|H_{n-1}\|_{\X_R,j})\|H_n-H_{n-1}\|_{\X_R,m-j}\\
    &\leq \lambda\|H_n-H_{n-1}\|_{\X_R,m} + \widetilde{K}_m\mu^{n-1},
\end{aligned}
\]
where we have set $a_m=c_5 m^d (m+1)^{(m+3)/2}\|\nabla g\|_{L^1}^m\|g\|_{W^{3,1}}$
(the coefficient in the first line of~\eqref{e:highnonlin}),
$\widetilde{K}_m = 2 C_0 (a_m K_0 +c_5 K_m) + 2 c_5 m^d \sum_{j=1}^{m-1} \binom{m}{j} K_j C_{m-j}$,
and the constants $K_j$ are given by the previous lemma.
By recurrence (notice that $\mu>\lambda$), it is easy to see that for every $n$,
\[
\begin{aligned}
  \|H_{n+1} - H_n\|_{\X_R,m}
    &\leq \lambda^{n+1}\|H_1 - H_0\|_{\X_R,m} + \widetilde{K}_m\bigl(\lambda^{n-1} + \lambda^{n-2}\mu + \dots + \mu^{n-1}\bigr)\\
    &\leq \bigl(\lambda K_m + \tfrac1{\mu-\lambda}\widetilde{K}_m\bigr)\mu^n,
\end{aligned}
\]
which concludes the induction. In conclusion, the sequence $(H_n)_{n\in\N}$
converges in all spaces $\X_R^m$.
\end{proof}

\end{document}